\documentclass[final,10pt]{amsart} 
\usepackage{latexsym,amssymb}
\usepackage{hyperref} 
\usepackage{amsmath,amsthm}
\usepackage{amsfonts} 
\usepackage[american]{babel}  
\usepackage{mathrsfs}
\usepackage{psfrag} 
\usepackage[dvips]{graphicx} 

\newcommand{\R}{\mathbb{R}} 

\newcommand{\C}{\mathbb{C}}

\newcommand{\eee}{\`e}

\numberwithin{equation}{section} 
\newtheorem{theorem}{Theorem}[section]
 
\newtheorem{cor}[theorem]{Corollary}
 
\newtheorem{remark}[theorem]{Remark}

\begin{document} 

\title[Stability for the Green's function] {Stability in an overdetermined problem for the Green's function}

\author[V. Agostiniani]{Virginia Agostiniani} 
\author[R. Magnanini]{Rolando Magnanini}

\address{V. Agostiniani, SISSA, via Beirut 2-4, 34151 Trieste - Italy} 
\email{vagostin@sissa.it}

\address{R. Magnanini, Dip.to di Matematica ``U. Dini'', Universit\`a 
degli Studi di Firenze, viale Morgagni 67/A, 50134 Firenze - Italy}
\email{rolando.magnanini@math.unifi.it}

\date{}

\begin{abstract} 
In the plane, we consider the
problem of reconstructing a domain from the normal derivative of its Green's
function (with fixed pole) relative to the 
Dirichlet problem for the Laplace operator. By means of the theory of 
conformal mappings, we derive stability
estimates of polynomial type in H\"older norms. 
\end{abstract}

\maketitle


\section{Introduction}\label{intro}

The study of overdetermined boundary problems in partial differential equations finds
its motivations in many areas of mathematics, such as inverse and free boundary problems, isoperimetric inequalities and optimal 
design. 
As in Serrin's seminal paper \cite{Se}, in many such problems the analysis is mainly focused on the (spherical) symmetry of the 
domain considered.  

In recent years, several authors have commenced to analyse the stability of the aforementioned symmetric configurations in presence of 
approximate (boundary) data (\cite{Af}, \cite{Bra}, \cite{Bra2}, \cite{MR}); see also the work on quantitative isoperimetric inequalities 
(\cite{C}, \cite{Mag}).

In \cite{Af}, a logarithmic estimate of approximate (spherical) symmetry is deduced for a quite general semilinear overdetermined 
problem. From the proof, based on an ingenious adaptation of Serrin's moving-planes argument, it is clear that the logarithmic 
character of the stability estimate is due to the use of Harnack inequality.  
Such a drawback appears to be inherent in the method employed and cannot even be removed by considering simpler nonlinearities.

An improved estimate -- of polynomial type, but only for the torsion problem -- has been derived in \cite{Bra} by combining 
Poho\u{z}aev integral identity and some geometric inequalities.

In the present paper, we will tackle a more detailed study of the stability in the plane by exploiting the theory of conformal mappings
as we have already done in \cite{Ag} for the study of symmetries, with the aim of deriving optimal estimates. 

As in \cite{Ag}, we will work on a case study: in a planar bounded domain $\Omega$ with boundary
$\partial\Omega$ of class $C^{1,\alpha}$, we shall consider the problem
\begin{eqnarray}
                            -\Delta u\!\!\!&=&\!\!\!\delta_{\zeta_o}\hspace{1.5cm}\mbox{in }\Omega, \label{pb1}\\ 
                                      u\!\!\!&=&\!\!\!0\hspace{1.5cm}\mbox{on }\partial\Omega, \label{pb2}\\ 
\frac{\partial u}{\partial\nu}\!\!\!&=&\!\!\!\varphi\hspace{1.5cm}\mbox{on }\partial\Omega. \label{pb3} 
\end{eqnarray} 
where $\nu$ is the \emph{interior} normal direction to $\partial\Omega$, $\delta_{\zeta_o}$ is the Dirac delta centered at a given 
point $\zeta_o\in\Omega$ and $\varphi:\partial\Omega\rightarrow\R$ is a positive given function of arclength, measured from a 
reference point on $\partial\Omega$. 

Problem (\ref{pb1})-(\ref{pb3}) should be interpreted as follows: find a domain $\Omega$ whose Green's function $u$ with pole at
$\zeta_o$ has gradient with values on the boundary that fit those of the given function $\varphi$.
This problem has some analogies to a model for the Hele-Shaw flow, as
presented in \cite{Gu} and \cite{Sa}.

In \cite{Ag}, we estabilished a connection between $\varphi$ and $\Omega$ by using conformal mappings:
chosen two distinct points $\zeta_b$ and $\zeta_o$ and a number $\alpha\in(0,1)$, we introduced the set  
$$
\mathscr O=\{\Omega\subseteq\C:\Omega\mbox{ open, bounded, simply connected, }C^{1,\alpha},\;
\zeta_o\in\Omega,\;\zeta_b\in\partial\Omega\}
$$ 
and the class of functions
$$
\mathscr F=\{f\in C^{1,\alpha}(\overline D,\C):f\mbox{ one-to-one, analytic, }f(0)=\zeta_o,\;f(1)=\zeta_b\},
$$
where $D$ is the open unit disk.
\par 
By the Riemann Mapping Theorem $\mathscr O$ and $\mathscr F$ are in one-to-one correspondence. 
In   ̃\cite[Theorem 2.2]{Ag}, we proved that the operator $\mathcal T$ that to each 
$f\in\mathscr F$ associates the interior normal derivative $\mathcal T(f)$ on $\partial\Omega$ of the solution of 
(\ref{pb1})-(\ref{pb2}) is injective: an $f\in\mathscr F$ is uniquely determined by the formula
\begin{equation}\label{rappresentazione f}
f'(z)=e^{i\gamma}\exp\left\{\frac 1{2\pi}\int_0^{2\pi}\frac{e^{it}+z}{e^{it}-z}\log\frac
1{2\pi\varphi(\Phi^{-1}(t))}dt\right\},\ \ z\in D, 
\end{equation} 
where 
\begin{equation}\label{Fi}
\varphi(s)=\mathcal T(f)(s),\ \ 
\Phi(s)=2\pi\int_0^s\varphi(\sigma)d\sigma,\ \ s\in[0,|\partial\Omega|];
\end{equation}
here, the constant $\gamma$ is normalized by the condition
\begin{equation}\label{normalizz.}
\int_0^1f'(t)dt=\zeta_b-\zeta_o.
\end{equation}
\par
By means of (\ref{rappresentazione f})-(\ref{normalizz.}), in \cite{Ag} we obtained results relating the symmetry 
of $\Omega$ to certain invariance properties of $\varphi$.

Here, by using the same ideas, we deduce stability results both near the disk 
and near any simply connected domain. Two typical results that better illustrate our work follow.

Preliminarly, we introduce some more notations. Given some positive constants
$L$, $m$, $M_0$ and $M_1$, we define two classes of functions:
$$ 
\mathscr G_0^L=\{\varphi\in C^{0,\alpha}(\R):\ \ \varphi\mbox{ is }L\mbox{-periodic},\ \ \varphi\geq m,
\ \ \Vert \varphi\Vert_{0,\alpha,[0,L]}\leq M_0\}, 
$$ 
$$ 
\mathscr G_1^L=\{\varphi\in\mathscr G_0^L\cap C^{1,\alpha}(\R):\ \ \Vert\varphi\Vert_{1,\alpha,[0,L]}\leq M_1\}.
$$
For the definitions of the relevant H\"older norms, we refer the reader to Section \ref{ingredienti}.

\begin{theorem}\label{raggi} 
Let $\Omega\in \mathscr O$ be 
with perimeter $L$.

Assume that $B(\zeta_o,\rho)$ and $B(\zeta_o,R)$ are  
the largest disk contained in $\Omega$ and the smallest disk containing $\Omega$,
centered at $\zeta_o$, respectively.

Let $\varphi$ be the interior normal derivative on $\partial\Omega$ (as function of the arclength) 
of the solution of \eqref{pb1}-\eqref{pb2}
and set $C=\frac 1{2\pi\rho}$.

Then, if $\varphi\in\mathscr G_0^L,$ there exists a constant $K$, depending on $\alpha$, $\rho$ and $M_0,$ such that 
\begin{equation*}
R-\rho\leq K\Vert\varphi-C\Vert_{0,\alpha,[0,L]}. 
\end{equation*}
\end{theorem} 
Theorem \ref{raggi} can be considered an analogue of   ̃\cite[Theorem 1]{Af} and   ̃\cite[Theorem 1.2]{Bra}.
Notice that here we obtain Lipschitz stability.
In the following result, we give a stability estimate involving the Hausdorff
distance $d_{\mathit H}$ of any
two bounded simply connected domains (for the 
definition of $d_{\mathit H}$, we refer the reader to Section \ref{stability}). 

\begin{theorem}\label{hausdorff}
Let $\Omega_1$ and $\Omega_2\in\mathscr O$ be domains with the same perimeter $L$ and $f_1$ and
$f_2$ the corresponding conformal mappings in $\mathscr F$. Suppose that $\mathcal T(f_1)$, 
$\mathcal T(f_2)\in\mathscr G_1^L$.

Then, up to rotations around $\zeta_o$, we have that 
\begin{equation*}
d_{\mathit H}(\Omega_1,\Omega_2)\leq 
K\Vert\mathcal T(f_1)-\mathcal T(f_2)\Vert_{1,0,[0,L]}^{\alpha},
\end{equation*}
where the constant $K$, whose expression can be deduced from the proof, depends on $\alpha$, $m$, $M_1$ and $L$.
\end{theorem}


Theorem \ref{hausdorff} seems to be new. Compare the Lipschitz stability obtained in Theorem \ref{raggi}
to the H\" older-type estimate obtained in Theorem \ref{hausdorff} (see Section \ref{stability} for the details). 
  In Subsection \ref{subsez.dom.gen} we will also present a more
general version of Theorem \ref{hausdorff}. 

\section{Some useful notations and results}\label{ingredienti}

In what follows, $D$ will always be the open unit disk in $\C$ centered at
$0.$ 

Let $\varphi:I\rightarrow\R$ be a function defined on an interval $I\subseteq\R$. We denote
$$
\Vert\varphi\Vert_{\infty,I}=\sup_{I}|\varphi|,\ \ 
[\varphi]_{k,\alpha,I}=\sup_{\stackrel{x,y\in I}{x\neq y}}
\frac{|\varphi^{(k)}(x)-\varphi^{(k)}(y)|}{|x-y|^{\alpha}}, 
$$
where
$k=0,1,...$, $0<\alpha\leq 1$ and $\varphi^{(k)}$ is the $k$-th derivative of $\varphi$, when defined. 
Moreover, we set:
$$
\Vert\varphi\Vert_{k,\alpha,I}=\sum_{j=0}^k\Vert\varphi^{(k)}\Vert_{\infty,I}+[\varphi]_{k,\alpha,I};\ \ 
\Vert\varphi\Vert_{k,0,I}=\sum_{j=0}^k\Vert\varphi^{(k)}\Vert_{\infty,I}
$$ 
and 
$$ C^{k,\alpha}(I)=\{\varphi\in C^0(I):\;\;\Vert\varphi\Vert_{k,\alpha,I}<+\infty\}. 
$$ 

Let us recall some basic facts (see \cite{go} and \cite{Ma} for more details). If $\Omega\subseteq\C$ is a simply connected domain 
bounded by a Jordan curve and $\zeta_o\in\Omega$, then, from the Riemann Mapping Theorem, it follows that $\Omega$ 
is the image of an analytic function $f:D\rightarrow\Omega$ which induces a homeomorphism between the
closures $\overline D$ and $\overline\Omega$, has non-zero derivative $f'$ in $D$ and is such that $f(0)=\zeta_o.$ 
An application of Schwarz's Lemma proves that $f$ is unique if it fixes a point of the boundary, say
$f(1)=\zeta_b$ for a certain $\zeta_b\in\partial\Omega.$ Moreover, if $\Omega$ is of class $C^{1,\alpha},$ for a certain
 $\alpha\in(0,1),$ then, by Kellogg's theorem, we can infer that $f\in C^{1,\alpha}(\overline D)$.

By keeping in mind the identification of the classes $\mathscr O$ and $\mathscr F$ introduced in Section \ref{intro}, let us
recall some formulas from \cite{Ag}, which will be useful in the sequel. Let $\mathcal T$ be the operator that associates to each 
$f$ in $\mathscr F$ the interior normal derivative $\frac{\partial u}{\partial\nu}$ 
--- as function of the arclength $s$, which will be measured counterclockwise on $\partial\Omega$ and starting from $\zeta_b$ --- of
the solution of (\ref{pb1})-(\ref{pb2}). We can define parametrically
the values of $\mathcal T(f)(s),$ $s\in[0,|\partial\Omega|],$ by
\begin{equation}\label{param} 
s = \int_0^{\theta}|f'(e^{it})|dt,\ \ \mathit T(f) = \frac
1{2\pi|f'(e^{i\theta})|},\ \ \theta\in[0,2\pi]. 
\end{equation} 
\par
Observe that  
$\mathcal T(f)$ is of class $C^{0,\alpha}$ and satisfies the compatibility conditions
$$
\int_0^{|\partial\Omega|}\mathcal{T}(f)(s)ds=1,\ \ \mathcal{T}(f)>0\mbox{ on }[0,|\partial\Omega|].
$$ 
From (\ref{param}), it descends the relation
\begin{equation*}
2\pi\mathcal T(f)(s(\theta))s'(\theta)=1,\ \ \theta\in[0,2\pi],
\end{equation*} 
which, once integrated, together with (\ref{param}), gives 
\begin{equation}\label{teta}
s(\theta)=\Phi^{-1}(\theta),\ \ |f'(e^{i\theta})|=\frac1{2\pi\mathcal T(f)(\Phi^{-1}(\theta))},
\ \ \theta\in[0,2\pi],
\end{equation}
where $\Phi^{-1}$ is the inverse of the function $\Phi$ defined in (\ref{Fi}).  

%

\section{Stability estimates}\label{stability}

For $\Omega_j\in\mathscr O$, we fix $\zeta_j\in\partial\Omega_j$ and let
$f_j$ be the mapping in $\mathscr F$ (with $\zeta_b=\zeta_j$) corresponding to $\Omega_j$ ($j=1,2$). 
From (\ref{rappresentazione f}) we know that 
\begin{equation}\label{rapprefj}  
f_j'(z)=e^{i\gamma_j}\exp\left\{\frac 1{2\pi}\int_0^{2\pi}\frac{e^{it}+z}{e^{it}-z}\log\frac
1{2\pi\mathcal T(f_j)(\Phi_j^{-1}(t))}dt\right\},\ \ z\in D, 
\end{equation}
for some $\gamma_j\in\R$, where $\Phi_j^{-1}$ is the inverse of the function
$\Phi_j$ defined by (\ref{Fi}) with $f$ replaced by $f_j$ ($j=1,2$).

We are going to estimate how far the domains $\Omega_1$ and $\Omega_2$ are from one other
(up to rotations), depending on an appropriate norm of the difference of the functions 
$\mathcal T(f_1)$ and $\mathcal T(f_2)$. 


\subsection{A preliminary estimate}

All our estimates will be based on Theorem \ref{stab.gen} below, where a bound of the norm
$$
\Vert f_1-f_2\Vert_{1,0,\partial D}
$$ 
is given in terms of the H\"older norm of the difference between the composite functions
$\mathcal T(f_1)\circ\Phi_1^{-1}$ and $\mathcal T(f_2)\circ\Phi_2^{-1}$, which are defined on $\partial D$
and not on $\partial\Omega_1$ and $\partial\Omega_2$. Later on, we shall convert such a bound into estimates
involving the functions $\mathcal T(f_j)$ ($j=1,2$) only.
\par
To this end, let us list here two estimates of H\"older seminorms which will be useful in the sequel.
Let us define
\begin{equation}\label{psi}
\psi_j=\mathcal T(f_j)\circ\Phi_j^{-1}\ \ (j=1,2),
\end{equation}
and
\begin{equation}\label{acca}
h=\log\psi_1-\log\psi_2.
\end{equation} 
If $\mathcal T(f_j)\in\mathscr G_0^{L_j}$, then
\begin{equation}\label{stimina}
[\psi_j]_{0,\alpha,[0,2\pi]}\leq\frac{[\mathcal T(f_j)]_{o,\alpha,[0,L_j]}}{(2\pi m)^{\alpha}}
\leq\frac{M_0}{(2\pi m)^{\alpha}}\ \ (j=1,2), 
\end{equation} 
and
\begin{equation}\label{stimacca} 
[h]_{0,\alpha,[0,2\pi]}\leq C_1\Vert\psi_1-\psi_2\Vert_{\infty,[0,2\pi]}
+C_2[\psi_1-\psi_2]_{0,\alpha,[0,2\pi]}, 
\end{equation}
where 
\begin{equation*}
C_1=\frac{M_0^2}{(2\pi)^{\alpha}m^{\alpha+3}},
\ \ C_2=\frac{M_0}{m^2}. 
\end{equation*}  
These two estimates follow from the general
fact that, if $\xi$ and $\eta$ are real-valued functions defined on intervals in $\R$,
then
$$
[\xi\circ\eta]\leq[\xi]_{0,\alpha}[\eta]_{0,1}^{\alpha},
$$
and from some algebraic identities.

\begin{theorem}\label{stab.gen} 
Given $\Omega_j\in\mathscr O$ ($j=1,2$), suppose that the arclength is measured counterclockwise on
$\partial\Omega_j$ starting from $\zeta_j\in\partial\Omega_j$ and assume that $f_j$ is 
the function in $\mathscr F$ (with $\zeta_b=\zeta_j$) corresponding to $\Omega_j$ ($j=1,2$).

Let $\psi_j$ be defined by (\ref{psi}) and suppose that $\mathcal T(f_j)\in\mathscr G_0^{L_j}$ ($j=1,2$).
Then, up to rotations around $\zeta_o$, we have that 
\begin{equation}\label{fond} 
\Vert f_1-f_2\Vert_{1,0,\partial D} \leq K\Vert\psi_1-\psi_2\Vert_{0,\alpha,\partial D}, 
\end{equation}
where $K$, whose expression can be deduced from the proof, is a constant depending on
$\alpha$, $m$ and $M_0$.
\end{theorem}

\begin{proof} 
Up to a rotation around $\zeta_o$, we can assume that in (\ref{rapprefj})
$\gamma_1=\gamma_2=\gamma$. Let us set
$$
\beta_j(z)=\arg f_j'(z),\ \ z\in D\ \ (j=1,2).
$$
It is clear that
\begin{eqnarray*}
|f'_1(z)-f'_2(z)| &=& \left|| f'_1(z)|e^{i\beta_1(z)}-|f'_2(z)|e^{i\beta_2(z)}\right|\\
                      &\leq& \left\Vert f'_1(z)|-|f'_2(z)|\right|+|f'_2(z)|\left|e^{i\beta_1(z)}-e^{i\beta_2(z)}\right|, 
\end{eqnarray*} 
and hence
\begin{equation}\label{c2}
|f'_1(z)-f'_2(z)|\leq\Vert f'_1(z)|-|f'_2(z)\Vert+|f'_2(z)\Vert\beta_1(z)-\beta_2(z)|,
\end{equation} 
since
$$
\left|e^{i\beta_1(z)}-e^{i\beta_2(z)}\right| \leq|\beta_1(z)-\beta_2(z)|.
$$ 
\par
Thus, by keeping in mind \eqref{rapprefj} and writing $z=re^{i\theta}$, it turns out that
\begin{eqnarray*} 
\beta_1(z)-\beta_2(z)&=& -\frac 1{2\pi}\int_0^{2\pi}\frac{2r\sin(\theta-t)} {1+r^2-2r\cos(\theta-t)}h(t)dt\\
                              &=& \frac 1{2\pi}\int_0^{\pi}\frac{2r\sin t} {1+r^2-2r\cos t}[h(\theta+t)-h(\theta-t)]dt,
\end{eqnarray*} 
where $h$ is defined as in (\ref{acca}). 
Since  
\begin{equation*}
 0\leq\frac{2r\sin t}
{1+r^2-2r\cos t}\leq\frac 1{\tan\frac t2},\ \ t\in[0,\pi],
\end{equation*}
and $h\in C^{0,\alpha}[0,2\pi]$, we get
\begin{equation}\label{c4}
|\beta_1(z)-\beta_2(z)|\leq\frac{[h]_{0,\alpha,[0,2\pi]}}{2\pi}\int_0^{\pi}\frac{(2t)^{\alpha}}{\tan\frac t2}dt,
\end{equation}
and the integral converges. 

On the other hand, 
$$
|f_j'(z)|=\exp\left\{\frac 1{2\pi}\int_0^{2\pi}\frac{1-r^2}{1+r^2-2r\cos(t)}\log\frac
1{2\pi\psi_j(t)}dt\right\},\ \ z\in D\ \ (j=1,2); 
$$
thus, 
\begin{equation}\label{c5}
|f_2'(z)|\leq\frac 1{2\pi m},
\end{equation} 
and 
\begin{equation}\label{c6}
\Vert f'_1(z)|-|f'_2(z)\Vert
                                \leq\frac 1{2\pi m^2}\Vert\psi_1-\psi_2\Vert_{\infty,[0,2\pi]},
\end{equation}
since both $\psi_1$ and $\psi_2$ are bounded below by $m$.
 
From (\ref{c2}), (\ref{c4}), (\ref{c5}) and (\ref{c6}), we infer that
\begin{equation*}
|f'_1(z)-f'_2(z)|\leq\frac 1{2\pi m^2}\Vert\psi_1-\psi_2\Vert_{\infty,[0,2\pi]}
              +\frac{c_{\alpha}}m[h]_{0,\alpha,[0,2\pi]},
\end{equation*} 
where 
$$
c_{\alpha}=\frac{2^{\alpha}}{4{\pi}^2}\int_0^{\pi}t^{\alpha}\cot\frac t2dt.
$$ 
Notice that
\begin{equation*}\label{daf'af}
|f_1(e^{i\theta})-f_2(e^{i\theta})|=\left|\int_0^1\frac d{dt}[f_1(te^{i\theta})-f_2(te^{i\theta})]dt\right|
\leq\Vert f_1'-f_2'\Vert_{\infty,\partial D}; 
\end{equation*} 
in order to obtain (\ref{fond}), we write
$$ 
\Vert f_1-f_2\Vert_{\infty,\overline D} \leq\frac 1{2\pi m^2}\Vert\psi_1-\psi_2\Vert_{\infty,[0,2\pi]}
+\frac{c_{\alpha}}m[h]_{0,\alpha,[0,2\pi]}
$$ 
and we estimate $[h]_{0,\alpha,[0,2\pi]}$ in terms of $\Vert\psi_1-\psi_2\Vert_{0,\alpha,[0,2\pi]}$,
by using (\ref{stimacca}).
\end{proof}


\subsection{Stability near a disk}\label{subsez.disco}

As we pointed out in \cite{Ag}, the disk is the only domain whose Green's function has constant normal 
derivative on the boundary. More precisely, in our notations, the mappings
\begin{equation}\label{fdisco}
 f_C(z)=\zeta_o+\frac {e^{i\gamma}}{2\pi C}z,\ \ z\in\overline D,
\end{equation}
with $\gamma\in\R$, are the only elements $\mathscr F$ such that $\mathcal T(f)=C$.
The next result specifies how far from $f_C$ is a mapping $f\in\mathscr F$ if $\mathcal T(f)$ is not constant. 

\begin{theorem}\label{disco} 
Let $f_C$ be given by (\ref{fdisco}) for some constants $C\in[m,M_0]$ and $\gamma\in\R$
and let $\Omega$ be in $\mathscr O$ with perimeter $L$. 

If $\mathcal T(f)\in\mathscr G_0^L$, then
\begin{equation*}
\Vert f-f_C\Vert_{1,0,\partial D} \leq K\left[1+\frac1{(2\pi m)^{\alpha}}\right]\Vert\mathcal T(f)-C\Vert_{0,\alpha,[0,L]},
\end{equation*} 
where $K$ is the constant of Theorem \ref{stab.gen}.
\end{theorem} 

\begin{proof} 
Theorem \ref{stab.gen} gives that 
\begin{equation*}
\Vert f-f_C\Vert_{1,0,\partial D}\leq K\Vert\psi-C\Vert_{0,\alpha,[0,2\pi]}, 
\end{equation*}
where $\psi=\mathcal T(f)\circ\Phi^{-1}$; it remains to estimate the right-hand side of the latter inequality by the 
H\"older norm of $\mathcal T(f)-C$. This is readly achieved by observing that 
$$
\Vert\psi-C\Vert_{\infty,[0,2\pi]}=\Vert\mathcal T(f)-C\Vert_{\infty,[0,L]}
$$
and, from (\ref{stimina}), that
$$
[\psi-C]_{0,\alpha,[0,2\pi]}\leq\frac1{(2\pi m)^{\alpha}}[\mathcal T(f)-C]_{0,\alpha,[0, L]}. 
$$
\end{proof}

We are now in position to prove Theorem \ref{raggi}.

\begin{proof}[Proof of Theorem \ref{raggi}]
Let $f_C$ be defined as in (\ref{fdisco}); then $\rho=|f_C(e^{i\theta})-\zeta_o|$ for every $\theta\in[0,2\pi]$.

Now, notice that
$$
R=\max_{0\leq\theta\leq2\pi}|f(e^{i\theta})-\zeta_o|;
$$
therefore, if $\theta_o\in[0,2\pi]$ maximizes $|f(e^{i\theta})-\zeta_o|$, we have:
\begin{equation*}
R-\rho=|f(e^{i\theta_o})-\zeta_o|-|f_C(e^{i\theta_o})-\zeta_o|\\
\leq|f(e^{i\theta_o})-f_C(e^{i\theta_o})|.
\end{equation*}
Thus, the conclusion plainly follows from Theorem \ref{disco}.
\end{proof}

\begin{remark}{\rm
In order to compare this result with   ̃\cite[Theorem 1]{Af} and   ̃\cite[Theorem 1.2]{Bra}, 
we observe that the proof of Theorem \ref{raggi} is straightforward also 
if we replace $B(\zeta_o,\rho)$ and $B(\zeta_o,R)$ by the largest disk
contained in $\Omega$ and for the smallest disk containing $\Omega$ (not necessarily centered in $\zeta_o$).
}
\end{remark}

\subsection{Domains with same perimeter}\label{subsez.perim}

We want to estimate the right-hand side of (\ref{fond}) in terms of some suitable distance between the functions 
$\mathcal T(f_1)$ and $\mathcal T(f_2)$. In this subsection, we shall start by considering the case of two domains with the same 
perimeter. Differently from Subsection \ref{subsez.disco}, it seems that in this case we cannot avoid to require that
$\mathcal T(f_1)$ and $\mathcal T(f_2)$ are of class $C^{1,\alpha}$.

\begin{theorem}\label{lugua} 
Given $\Omega_1$, $\Omega_2\in\mathscr O$, both with perimeter that equals $L$, 
let $f_1$ and $f_2$ be the conformal mappings in $\mathscr F$ corresponding to $\Omega_1$
and $\Omega_2$, respectively.

If $\mathcal T(f_1)$, $\mathcal T(f_2)\in\mathscr G_1^L$, 
then, up to domains' rotations around $\zeta_o$, we have that
\begin{equation}\label{stessiperim}
\Vert f_1-f_2\Vert_{1,0,\partial D} \leq K\left\{\Vert\varphi_1-\varphi_2\Vert_{\infty,[0,L]}^{\alpha}
+\Vert\varphi_1'-\varphi_2'\Vert_{\infty,[0,L]}\right\}, 
\end{equation}
where $\varphi_j=\mathcal T(f_j)$ ($j=1,2$) and the constant $K$ depends on $\alpha$, $m$, $M_1$ and $L$ and can
be deduced from the proof. 
\end{theorem}

\begin{proof} 
From (\ref{param}) and (\ref{teta}), we have that 
$$ 
\frac{\theta}{2\pi}=\int_0^{\Phi_1^{-1}(\theta)}\varphi_1(\sigma)d\sigma=
\int_0^{\Phi_2^{-1}(\theta)}\varphi_2(\sigma)d\sigma,\ \ \theta\in[0,2\pi] ,
$$ 
and hence
$$
\int_{\Phi_1^{-1}(\theta)}^{\Phi_2^{-1}(\theta)}\varphi_1(\sigma)d\sigma=
\int_0^{\Phi_2^{-1}(\theta)}[\varphi_1(\sigma)-\varphi_2(\sigma)]d\sigma. 
$$ 
\par
Thus,
\begin{equation}\label{mags}
|\Phi_1^{-1}(\theta)-\Phi_2^{-1}(\theta)|\leq\frac Lm\Vert\varphi_1-\varphi_2\Vert_{\infty,[0,L]},\ \ \theta\in[0,2\pi],
\end{equation}
since $\varphi_1,$ $\varphi_2\in\mathscr G_1^L$.

Let $\psi_j$ be the functions defined in (\ref{psi}).
We now estimate $\Vert\psi_1-\psi_2\Vert_{\infty,[0,2\pi]}$ and $[\psi_1-\psi_2]_{0,\alpha,[0,2\pi]}$. 
It is clear that
\begin{multline*}
|\psi_1(\theta)-\psi_2(\theta)| \leq \label{punto2}\\ 
|\varphi_1(\Phi_1^{-1}(\theta))-\varphi_1(\Phi_2^{-1}(\theta))|+
|\varphi_1(\Phi_2^{-1}(\theta))-\varphi_2(\Phi_2^{-1}(\theta))|\leq \\ 
[\varphi_1]_{0,\alpha,[0,L]}
|\Phi_1^{-1}(\theta)-\Phi_2^{-1}(\theta)|^{\alpha}+
\Vert\varphi_1-\varphi_2\Vert_{\infty,[0,L]},
\end{multline*}
and hence, from (\ref{mags}), we obtain the inequality 
\begin{equation}\label{tria}
\Vert\psi_1-\psi_2\Vert_{\infty,[0,2\pi]}\leq
\left[M_1\left(\frac Lm\right)^{\alpha}+(2M_1)^{1-\alpha}\right]
\Vert\varphi_1-\varphi_2\Vert_{\infty,[0,L]}^{\alpha}, 
\end{equation}
since $\varphi_1,$ $\varphi_2\in\mathscr G_1^L$.

Next, by Lagrange's Theorem, we have:
\begin{eqnarray*}
|(\psi_1-\psi_2)(\theta)-(\psi_1-\psi_2)(\hat\theta)| &\leq& \Vert\psi_1'-\psi_2'\Vert_{\infty,[0,2\pi]}|\theta-\hat\theta|\\                                                                             
                                                                          &\leq& (2\pi)^{1-\alpha}
                                                                                     \Vert\psi_1'-\psi_2'\Vert_{\infty,[0,2\pi]}|\theta-\hat\theta|^{\alpha};
\end{eqnarray*}
thus,
\begin{equation}\label{siga}
[\psi_1-\psi_2]_{0,\alpha,[0,2\pi]}\leq(2\pi)^{1-\alpha}\Vert\psi_1'-\psi_2'\Vert_{\infty,[0,2\pi]}.
\end{equation}                                                                                     
 
In order to estimate the right-hand side of the latter inequality, we notice that
$$
\psi_j'(\theta)=\frac{\varphi_j'(\Phi_j^{-1}(\theta))}{2\pi\varphi_j(\Phi_j^{-1}(\theta))}
\ \ (j=1,2)
$$
and, by setting $s_j=\Phi_j^{-1}(\theta)$, we write:
\begin{multline*}
2\pi|\psi_1'(\theta)-\psi_2'(\theta)| \leq\\ 
\frac{|\varphi_1'(s_1)-\varphi_1'(s_2)|}{\varphi_1(s_1)}
                                                             +\left|\frac 1{\varphi_1(s_1)}-\frac 1{\varphi_2(s_2)}\right|\varphi_1'(s_2)
                                                             +\frac{|\varphi_1'(s_2)-\varphi_2'(s_2)|}{\varphi_2(s_2)}\\
                                                  \leq  \frac{M_1}m|s_1-s_2|^{\alpha} +\frac{M_1}{m^2}
                                                              \Vert\psi_1-\psi_2\Vert_{\infty,[0,2\pi]} 
                                                              +\frac 1m\Vert\varphi_1'-\varphi_2'\Vert_{\infty,[0,L]}, 
\end{multline*}  
since $\varphi_1,$ $\varphi_2\in\mathscr G_1^L$.
By (\ref{mags}) and (\ref{tria}), we then obtain 
\begin{multline}\label{triatria}
2\pi\Vert\psi_1'-\psi_2'\Vert_{\infty,[0,2\pi]}\leq\frac{M_1}m\left(\frac
Lm\right)^{\alpha}\Vert\varphi_1-\varphi_2\Vert_{\infty,[0,L]}^{\alpha}\\ 
+\frac{M_1}{m^2}\left\{M_1\left(\frac Lm\right)^{\alpha}+(2M_1)^{1-\alpha}\right\}
\Vert\varphi_1-\varphi_2\Vert_{\infty,[0,L]}^{\alpha}
+ \frac 1m\Vert\varphi_1'-\varphi_2'\Vert_{\infty,[0,L]}. 
\end{multline}
Therefore, (\ref{stessiperim}) easily follows from (\ref{tria}) and from (\ref{siga}) 
together with the latter inequality.
\end{proof}

We recall that the {\it Hausdorff distance} between two compact subsets $A$ and $B$ of $R^{n}$ 
is defined as 
$$
d_{\mathit H}(A,B)=\max\{\rho(A,B),\rho(B,A)\},
$$
where
$$
\rho(A,B)=\sup_{a\in A}\inf_{b\in B}|a-b|.
$$

\begin{proof}[Proof of Theorem \ref{hausdorff}]
As usual, let $f_j$ be the mapping in $\mathscr F$ corresponding to $\Omega_j$ ($j=1,2$). Thus,
\begin{eqnarray*}
\rho(\Omega_1,\Omega_2) &=& \sup_{\zeta_1\in\Omega_1}\inf_{\zeta_2\in\Omega_2}|\zeta_1-\zeta_2|\\
                                         &=& \sup_{0\leq\theta_1\leq2\pi}
                                                 \inf_{0\leq\theta_2\leq2\pi}|f_1(e^{i\theta_1})-f_2(e^{i\theta_2})|,
\end{eqnarray*}
and hence
\begin{equation*}
\rho(\Omega_1,\Omega_2)\leq\sup_{0\leq\theta_1\leq2\pi}|f_1(e^{i\theta_1})-f_2(e^{i\theta_1})|
\leq\Vert f_1-f_2\Vert_{1,0,\partial D}.
\end{equation*}
The conclusion then follows from Theorem \ref{lugua}.
\end{proof}

\subsection{Generic domains}\label{subsez.dom.gen}

If $\Omega_1$ and $\Omega_2$ have different perimeters, say $L_1$ and $L_2$, the functions 
$\mathcal T(f_1)$ and $\mathcal T(f_2)$ are defined on different intervals, $[0,L_1]$ and $[0,L_2]$,
and we cannot compare their values directly. Thus, we rescale them: if $\varphi_j=\mathcal T(f_j)$, 
we set 
\begin{equation}\label{cappuccio}
\hat\varphi_j(s)=\varphi_j\left(\frac{L_j}Ls\right),\ \ s\in[0,L],\ \ \mbox{where}\ \ L=\frac{L_1+L_2}2\ \ (j=1,2).  
\end{equation}
The functions $\hat\varphi_j$ are now defined on a common interval.

\begin{theorem}\label{ultimo} 
Let $\Omega_1$ and $\Omega_2$ be domains in $\mathscr O$ with perimeters $L_1$ and
$L_2$, respectively, such that
$$
0<p\leq L_1,L_2\leq P
$$
for some constants $p$ and $P$. Let $f_1$, $f_2\in\mathscr F$ be as usual and assume that
(\ref{cappuccio}) holds.

If $\varphi_j=\mathcal T(f_j)\in\mathscr G_1^{L_j}$ ($j=1,2$), then, up to domains' 
rotations around $\zeta_o$, we have that
\begin{equation*}
\Vert f_1-f_2\Vert_{1,0,\partial D} \leq
K\left\{\left(\frac{|L_1-L_2|}P+\frac{\Vert\hat\varphi_1-\hat\varphi_2\Vert_{\infty,[0,L]}}{M_1}\right)^{\alpha}+
\Vert\hat\varphi_1'-\hat\varphi_2'\Vert_{\infty,[0,L]}\right\}, 
\end{equation*}
where the constant $K$ depends on $\alpha$, $m$, $M_1$, $p$ and $P$
and its expression can be deduced from the proof. 
\end{theorem}

\begin{proof}
We preliminary notice that 
\begin{equation}\label{older1}
\Vert\hat\varphi_j\Vert_{\infty,[0,L]}=\Vert\varphi_j\Vert_{\infty,[0,L_j]},
\ \ [\hat\varphi_j]_{0,\alpha,[0,L]}=\left(\frac{L_j}L\right)^{\alpha}[\varphi_j]_{0,\alpha,[0,L_j]}
\end{equation} 
and 
\begin{equation}\label{older2}
\Vert\hat\varphi_j'\Vert_{\infty,[0,L]}=\frac{L_j}L\Vert\varphi_j'\Vert_{\infty,[0,L_j]},
\ \ [\hat\varphi_j']_{0,\alpha,[0,L]}=\left(\frac{L_j}L\right)^{\alpha+1}[\varphi_j']_{0,\alpha,[0,L_j]}.
\end{equation}

The proof will proceed as the one of Theorem \ref{lugua}, with some variations. 
The following notations and formulas will be useful:
\begin{equation*}
\hat s_j(\theta)=\frac L{L_j}\Phi_j^{-1}(\theta),\ \
\psi_j(\theta)=\hat\varphi_j(\hat s_j(\theta))\ \ (j=1,2).
\end{equation*} 
Since
$$
\frac{\theta}{2\pi}=\int_0^{\Phi_j^{-1}(\theta)}\varphi_j(\sigma)d\sigma=\frac{L_j}L\int_0^{\hat
s_j(\theta)}\hat\varphi_j(\sigma)d\sigma,\ \ \theta\in[0,2\pi]\ \ (j=1,2),
$$ 
we derive an estimate similar to (\ref{mags}):
\begin{equation*} 
\frac{L_1}Lm|\hat s_1(\theta)-\hat s_2(\theta)|\leq
|L_1-L_2|\,\Vert\hat\varphi_1\Vert_{\infty,[0,L]}+L_2\Vert\hat\varphi_1-\hat\varphi_2\Vert _{\infty,[0,L]};
\end{equation*}
thus,
\begin{equation*}
|\hat s_1(\theta)-\hat s_2(\theta)|\leq\frac{M_1}m\frac{P^2}p\left\{\frac{|L_1-L_2|}P
+\frac{\Vert\hat\varphi_1-\hat\varphi_2\Vert_{\infty,[0,L]}}{M_1}\right\},\ \ \theta\in[0,2\pi]. 
\end{equation*} 
From now on, we can proceed, by using (\ref{older1}) and (\ref{older2}), as in the proof of Theorem (\ref{lugua}),  
with $\varphi_j$ and $s_j$ replaced by $\hat\varphi_j$
and $\hat s_j$, respectively: (\ref{tria}) changes into
\begin{equation}\label{nomina}
\Vert\psi_1-\psi_2\Vert_{\infty,[0,2\pi]}\leq K_1\left\{\frac{|L_1-L_2|}P
+\frac{\Vert\hat\varphi_1-\hat\varphi_2\Vert_{\infty,[0,L]}}{M_1}\right\}^{\alpha},
\end{equation}
where
$$
K_1=M_1\left[\left(\frac{M_1}m\frac{P^3}{p^2}\right)^{\alpha}+4^{1-\alpha}\right];
$$
(\ref{triatria}) becomes
$$
2\pi\Vert\psi_1'-\psi_2'\Vert_{\infty,[0,2\pi]}\leq K_2\left\{\frac{|L_1-L_2|}P
+\frac{\Vert\hat\varphi_1-\hat\varphi_2\Vert_{\infty,[0,L]}}{M_1}\right\}^{\alpha}
+\frac P{pm}\Vert\hat\varphi_1'-\hat\varphi_2'\Vert_{\infty,[0,L]}
$$
where 
$K_2$, easy computable, is still a constant depending on $\alpha$, $m$, $M_1$, $p$ and $P$.
The conclusion then follows from (\ref{nomina}), (\ref{siga}) and the latter inequality.
\end{proof}

By the same arguments used for the proof of Theorem \ref{hausdorff}, Theorem \ref{ultimo}
yields the following corollary.

\begin{cor}
Under the same assumptions of Theorem \ref{ultimo}, it holds that
$$
d_{\mathit H}(\Omega_1,\Omega_2)\leq
K\left\{\left(\Vert\hat\varphi_1-\hat\varphi_2\Vert_{\infty,[0,L]}+|L_1-L_2|\right)^{\alpha}+
\Vert\hat\varphi_2'-\hat\varphi_1'\Vert_{\infty,[0,L]}\right\},
$$
where $K$ is a constant depending on $\alpha$, $m$, $M_1$, on $p$ and $P$.
\end{cor}


\end{document}